\documentclass [10pt]{amsart}
\usepackage{amsfonts,tikz}
\usepackage{amsmath, amsthm, amssymb, longtable}
\usepackage{graphicx}
\usepackage{indentfirst}
\usepackage[T1]{fontenc}
\usepackage[margin=1in]{geometry}
\usepackage{color}
\usepackage{slashbox}

\numberwithin{equation}{section}

\newcommand{\Z}{\mathbb{Z}}
\newcommand{\C}{\mathbb{C}}

\newcommand{\N}{\mathbb{N}}
\newcommand{\R}{\mathbb{R}}

\newcommand{\Hc}{\mathcal{H}}

\newcommand{\etaq}[2]{\eta^{#2}( #1 z)}
\newcommand{\floor}[1]{\left\lfloor #1 \right\rfloor}
\newcommand{\ceiling}[1]{\left\lceil #1 \right\rceil}

\DeclareMathOperator{\SL}{SL}
\DeclareMathOperator{\sgn}{sgn}

\newtheorem{theorem}{Theorem}[section]
\newtheorem*{theorem*}{Theorem}

\newtheorem{coro}[theorem]{Corollary}
\newtheorem{defn}[theorem]{Definition}
\newtheorem{example}[theorem]{Example}
\newtheorem{lemma}[theorem]{Lemma}

\title{Counting Eta-Quotients of Prime Level}
\author{Allison Arnold-Roksandich, Kevin James, and Rodney Keaton}
\address{Department of Mathematics, Oregon State University,
Corvallis, OR 97331-4605, USA}
\email{arnoldra@math.oregonstate.edu}
\address{Department of Mathematics, Clemson University, Clemson, SC 29634}
\email{kevja@clemson.edu}
\address{Department of Mathematics and Statistics, East Tennessee State University, Johnson City, TN 37614-1710,USA}
\email{keatonr@etsu.edu}
\date{\today}

\begin{document}
\thanks{This study was supported by the NSF sponsored REU grant DMS-1156761}
\maketitle
\begin{abstract}
It is known that all modular forms on $\SL_2(Z)$ can be expressed as a rational function in $\etaq{}{}$, $\etaq{2}{}$ and $\etaq{4}{}$. By using a theorem found in \cite{Ono1}, and calculating the order of vanishing, we can compute the eta-quotients for a given level. Using this count, knowing how many eta-quotients are linearly independent and using the dimension formula, we can figure out a subspace spanned by the eta-quotients. In this paper, we primarily focus on the case where $N=p$ a prime.
\end{abstract}
\section{Introduction and Statement of Results}
Modular forms and cusp forms encode important arithmetic information, and are therefore important to study. An easy way to accomplish this is to study the Dedekind eta-function \begin{equation}\label{etadef}
\eta(z) := q^{1/24}\prod_{n\geq 1}(1-q^n),\quad\text{where }q=e^{2\pi iz}.
\end{equation} In particular, we focus on functions of the form \begin{equation}\label{etaqdef}f(z) = \prod_{d|N} \eta^{r_d}(d z),\quad r_d\in\Z\end{equation} which we call eta-quotients, as they provide nice examples of modular forms.

The following theorem is the primary motivation behind this paper.
\begin{theorem}\cite[Thm. 1.67]{Ono1} 
Every modular form on $\SL_2(\Z)$ may be expressed as a rational function in $\eta(z)$, $\eta(2z)$, and $\eta(4z)$.
\end{theorem}
While the recent work of Rouse and Webb (\cite{Rou}) has shown that Theorem 1.1 does not generalize to all levels, the subspace of eta-quotients for fixed level at least 2 is still interesting. The goal of this paper is to look at the vector space of modular forms with prime level, $M_k(\Gamma_1(p))$, and count the number of eta-quotients for fixed weight, $k$, and level $p$, and compare the span of these eta-quotients with $M_k(\Gamma_1(p))$. In other words, this paper focuses on explicitly counting the eta-quotients that are modular forms for the congruence subgroups $\Gamma_0(p)$ and $\Gamma_1(p)$ where $p$ is a prime.

\begin{theorem}\label{kcond}
Let $p>3$ be a prime and $k$ be an integer. Then there exists $f(z) = \eta^{r_1}(z)\eta^{r_p}(pz)$ such that $f(z)$ is a weakly holomorphic modular form with weight $k$ of level $\Gamma_1(p)$ if and only if $k$ is divisible by $h = \frac{1}{2}\gcd(p-1,24)$. 
\end{theorem}

This first theorem provides a condition on $k$ that is necessary and sufficient for showing that the space of weakly holomorphic weight $k$ and level $\Gamma_0(p)$ contains eta-quotients. With some effort, we can create similar conditions to guarantee when $f(z)$ is in $M_k(\Gamma_0(p))$. The next theorem gives an explicit count of the number of eta-quotients that are cusp forms of weight $k$ and level $\Gamma_0(p)$.

\begin{theorem}\label{etaPcount}
Let $p>3$ be a prime. Let $k=hk'$ where $h$ is the needed divisor of $k$ given by Theorem \ref{kcond}. Let $c$ be the smallest positive integer representative of $\frac{k'h}{12}$ modulo $\frac{p-1}{2h}$.
\begin{enumerate}
\item For $c = \frac{k(p+1)}{12} - \floor{\frac{k(p+1)}{12\frac{p-1}{2h}}} \frac{p-1}{2h}$, the number of eta-quotients in $S_k(\Gamma_1(p))$ is \[\frac{k(p+1)}{12d}-1.\]
\item For $c < \frac{k(p+1)}{12} - \floor{\frac{k(p+1)}{12\frac{p-1}{2h}}} \frac{p-1}{2h}$, the number of eta-quotients in $S_k(\Gamma_1(p))$ is \[ \ceiling{\frac{k(p+1)}{12\frac{p-1}{2h}}}.\]
\item For $c > \frac{k(p+1)}{12} - \floor{\frac{k(p+1)}{12\frac{p-1}{2h}}} \frac{p-1}{2h}$, the number of eta-quotients in $S_k(\Gamma_1(p))$ is \[ \floor{\frac{k(p+1)}{12\frac{p-1}{2h}}}.\]
\end{enumerate}
\end{theorem}

There are also eta-quotients in $M_k(\Gamma_1(p))$ that are not cusp forms that are given by the following theorem.

\begin{theorem}\label{noncuspCount}
Let $p>3$ be a prime. Then, $M_k(\Gamma_1(p))\setminus S_k(\Gamma_1(p))$ contains at least one eta-quotient if and only if $\frac{p-1}{2} | k$. Furthermore, for $k>0$ and $\frac{p-1}{2}|k$, there are exactly two eta-quotients in $M_k(\Gamma_1(p))\setminus S_k(\Gamma_1(p))$, which are of the form \[ \frac{\etaq{p}{\frac{2pk}{p-1}}}{\etaq{}{\frac{2k}{p-1}}},\quad\text{and}\quad \frac{\etaq{}{\frac{2pk}{p-1}}}{\etaq{p}{\frac{2k}{p-1}}}.\]
\end{theorem}

Finally, the following theorem also tells us the size of the subspace spanned by eta-quotients.

\begin{theorem}\label{plinind}
Let $p>3$ be a prime. Then, the eta-quotients in $M_k(\Gamma_1(p))$ given by the previous theorems are linearly independent.
\end{theorem}

Section 2 of this paper provides the necessary background for the results. The background includes information on modular forms, the dimension formula, and eta-quotients. Section 3 provides the proofs of the results given in this section. Finally, Section 4 details still open questions and some ideas of how to extend these results further.

\section{Background}
\subsection{Modular Forms}\label{ModForm}
In this section we present some definitions and basic facts from the theory of modular forms. For further details, the interested reader is referred to \cite[Ch. 3]{Kob1}.

\begin{defn}The \emph{modular group}, denoted $\SL_2(\Z )$, is the group of all matrices of determinant 1 and having integral entries.\end{defn} 

The modular group acts on the upper half plane, $\Hc = \{ x+iy | x,y\in \R, y>0\}$, by linear fractional transformations \[\begin{pmatrix}a & b\\c&d\end{pmatrix}z = \frac{az+b}{cz+d}.\]
Furthermore, if we define $\mathcal{H}^{*}$ to be the set $\mathcal{H}\cup\mathbb{Q}\cup\{i\infty\}$, then the action of $\SL_2(\mathbb{Z})$ on $\mathcal{H}$ extends to an action on $\mathcal{H}^{*}$ \cite{Kob1}. 

There are only certain specific subgroups of $\SL_2(\Z)$ which we will use for our purposes. They are 
\begin{align*}
\Gamma_0 (N) &= \left\lbrace \begin{pmatrix}a & b\\c&d\end{pmatrix} \in \SL_2(\Z)\left| \begin{pmatrix}a & b\\c&d\end{pmatrix} \equiv \begin{pmatrix}* & *\\0&*\end{pmatrix}\right. \pmod{N} \right\rbrace,\\
\Gamma_1 (N) &= \left\lbrace \begin{pmatrix}a & b\\c&d\end{pmatrix} \in \SL_2(\Z)\left| \begin{pmatrix}a & b\\c&d\end{pmatrix} \equiv \begin{pmatrix}1 & *\\0&1\end{pmatrix}\right. \pmod{N} \right\rbrace .\\
\end{align*}
Each of these subgroups is called a congruence subgroup of level $N$. Note, if $N=1$, then $\Gamma_0(N)=\Gamma_1(N)=\SL_2(\mathbb{Z})$. This brings us to our next definition. 
\begin{defn}
Let $\Gamma\leq\SL_2(\mathbb{Z})$ be a congruence subgroup and define an equivalence relation on $\mathbb{Q}\cup\{\infty\}$ by $z_1\sim z_2$ if there is a $\gamma\in\Gamma$ such that $\gamma\cdot z_1=z_2$. We call each equivalence class under this relation a \emph{cusp of} $\Gamma$.
\end{defn}
Now, for an integer $k$ and a function $f:\Hc^{*}\rightarrow\mathbb{C}$ and a $\gamma=\begin{pmatrix}a&b\\c&d\end{pmatrix}\in\SL_2(\mathbb{Z})$ we define the weight $k$ slash operator by
$$f|_k\gamma(z)=(cz+d)^{-k}f(\gamma\cdot z).$$
Note, we will often suppress the weight from the notation when it is clear from context, or irrelevant for our purposes.

We are now prepared to define the objects which will be of primary interest to us.
\begin{defn} A function $f: \Hc^{*} \rightarrow \C$ is called a \emph{weakly modular function} of weight $k$ and level $\Gamma$ if 
\begin{enumerate}
\item $f$ is holomorphic on $\Hc$,
\item $f$ is modular, i.e., for every $\gamma \in \Gamma$ and $z\in\Hc$ we have $f|\gamma(z)=f(z)$,
\item $f$ is meromorphic at each cusp of $\Gamma$.
\end{enumerate}
Furthermore, if we replace condition 3 by $f$ is holomorphic at each cusp of $\Gamma$ then we call $f$ a modular form. If we further replace condition 3 with $f$ vanishes at each cusp of $\Gamma$ then we call $f$ a cusp form.
\end{defn}

Consider a form, $f$, of level $N$. We will clarify what we mean by a function being ``holomorphic at a cusp''. First, consider the cusp $\{i\infty\}$, which we call ``the cusp at $\infty$''. Note, the matrix \[T=\begin{pmatrix}1&1\\0&1\end{pmatrix}\] is an element of $\Gamma_1(N)$ and hence $\Gamma_0(N)$ for every $N$, and as our function satisfies Condition 2, we have
$f(Tz)=f(z+1)=f(z)$, i.e. our function is periodic. It is a basic fact from complex analysis that such a function has a Fourier expansion of the form 
$$f(z)=\sum_{n=-\infty}^{\infty}a_nq^n\text{, where }q:=e^{2\pi iz}.$$
Using this, we say that $f$ is meromorphic at $\{i\infty\}$ if there is some $c<0$ such that $a_n=0$ for all $n<c$. We say that $f$ holomorphic at $\{i\infty\}$ if $a_n=0$ for all $n<0$, and we say that $f$ vanishes at $\{i\infty\}$ if $a_n=0$ for all $n\leq 0$. We call the smallest $n$ such that $a_n\neq 0$ the order of vanishing on the cusp at $\infty$. To cover another cusp, $\alpha$, let $\gamma\in \SL_2(\Z)$ be such that $\gamma\cdot\infty = \alpha$. Then, we need \[(cz+d)^{-k}f(\gamma\cdot z) = \sum_{n=-\infty}^{\infty} c_nq^n.\] If this holds, then we say $f$ is meromorphic at $\alpha$ if there is some $c<0$ such that $c_n=0$ for all $n<c$. We also similarly say that the smallest $n$ such that $c_n\neq 0$ is the order of vanishing at $\alpha$.

Now, we set some notation which we will use throughout. For $\Gamma\leq\SL_2(\mathbb{Z})$ we denote the space of weakly modular functions (modular forms, cusp forms, resp.) of level $\Gamma$ and weight $k$ by $M^{!}_k(\Gamma)$ ($M_k(\Gamma),S_k(\Gamma)$, resp.). Note, the spaces $S_k(\Gamma)\leq M_k(\Gamma)$ are finite dimensional complex vector spaces.

Throughout, we will also need the notion of a modular form with an associated character. We define a Dirichlet character of modulus $N$ as a map $\chi:\Z\rightarrow \C$ such that\begin{enumerate}
\item $\chi(m) = \chi(m+N)$ for all $m\in\Z$
\item If $\gcd(m,N)>1$ then $\chi(m) = 0$. If $\gcd(m,N) = 1$, then $\chi(N)\neq 0$.
\item $\chi(mn) = \chi(m)\chi(n)$ for all integers $m,n$.
\end{enumerate} Furthermore, if we let $c$ be the minimal integer such that $\chi$ factors through $\left(\mathbb{Z}/c\mathbb{Z}\right)^{\mathsf{x}}$, then we say $\chi$ has conductor $c$. 

Let $f\in M_k(\Gamma_1(N))$ and suppose further that $f$ satisfies
$$f|\gamma(z)=\chi(d)f(z)\text{, for all }\gamma=\begin{pmatrix}a&b\\c&d\end{pmatrix}\in\Gamma_0(N).$$
Then, we say that $f$ is a modular form of level $N$ and character $\chi$, and we denote the space of such functions by $M_k(N,\chi)$. Note, this is defined similarly for weakly modular functions and cusp forms.

It is well known that we have the following decomposition
$$M_k(\Gamma_1(N))=\!\!\!\!\!\bigoplus_{\chi\!\!\!\!\!\pmod{N}}\!\!\!\!\!M_k(N,\chi),$$
where our direct sum is over all Dirichlet characters modulo $N$. Decomposing further we have
$$M_k(N,\chi)=S_k(N,\chi)\oplus E_k(N,\chi)^{\perp},$$
into the space of cusp forms and its orthogonal complement with respect to the Petersson inner product, denoted $E_k(N,\chi)$, which is called the Eisenstein subspace.
\subsection{Dimension Formulas}
In this section we present formulas for the dimension of spaces of cusp and modular forms. For more details regarding dimension formulas, the interested reader is referred to \cite{Ste1}.

\subsubsection{The dimension formula for level $\Gamma_0(p)$ with trivial character}\label{Gamma0}
In this section we present a formula for the dimension of  $E_k(\Gamma_0(p))$ and $S_k(\Gamma_0(p))$ for $p\geq5$ a rational prime.

First, we set
$$\mu_{0,2}(p)=\left\{\begin{array}{lcl}0&\text{if}&p\equiv 3\pmod{4}\\2&\text{if}&p\equiv 1\pmod{4},\end{array}\right.$$
$$\mu_{0,3}(p)=\left\{\begin{array}{lcl}0&\text{if}&p\equiv 2\pmod{3}\\2&\text{if}&p\equiv 1\pmod{3}.\end{array}\right.$$
Then define
$$g_0(p)=\frac{p+1}{12}-\frac{\mu_{0,2}(p)}{4}-\frac{\mu_{0,3}(p)}{3}.$$
Using this we have $\dim S_2(\Gamma_0(p))=g_0(p)$ and $\dim E_2(\Gamma_0(p))=1$,
and for $k\geq 4$ even we have $\dim E_k(\Gamma_0(p))=2$ and
\begin{align*}
\dim S_k(\Gamma_0(p))=(k-1)(g_0(p)-1)+(k-2)+\mu_{0,2}(p)\lfloor k/4\rfloor+\mu_{0,3}(p)\lfloor k/3\rfloor.
\end{align*}
From this, we see that our formula depends on the congruence class which $k$ and $p$ lie in modulo 12, so compiling these different congruences together we have Table \ref{trivtab} in the Appendix.

\subsubsection{The dimension formula for $\Gamma_0(p)$ with quadratic character}
In this section we will consider the case that our level is $\Gamma_0(p)$ for some rational prime $p$ and that our associated character is quadratic. Note, at the end of the section we compile all of our computations together in a table for convenience.

In Section \ref{Gamma0} we considered the trivial character case, so we now set $\chi(\cdot)=\left(\frac{\cdot}{p}\right)$. We must compute the summations
$$\sum_{x\in A_4(p)}\chi(x),\sum_{x\in A_3(p)}\chi(x),$$ where $A_4(N) = \{ x\in \Z/N\Z: x^2+1=0\}$ and $A_3(p) = \{ x\in \Z/p\Z: x^2+x+1=0\}$.

First, we will consider $\sum_{x\in A_4(p)}\chi(x)$. This is clearly zero if $A_4(p)$ is empty, which occurs precisely when $p\equiv 3\pmod{4}$. Also, it is immediate that our summation equals $1$ when $p=2$.  Now suppose $p\equiv 1\pmod{4}$. Then $\# A_4(p)=2$. Note, if $r\in A_4(p)$ then $-r\in A_4(p)$ and  $\chi(r)=\chi(-r)$ since $\chi(-1)=1$. Furthermore, it is not hard to see that $\chi(r)=1$ if and only if there is an element of order 8 in $\left(\mathbb{Z}/p\mathbb{Z}\right)^{\mathsf{x}}$, i.e., $p\equiv 1\pmod{8}$. Thus, we have
$$\sum_{x\in A_4(p)}\chi(x)=\left\{\begin{array}{lcl}1&\text{if}&p=2\\0&\text{if}&p\equiv 3\pmod{4}\\2&\text{if}&p\equiv 1\pmod{8}\\-2&\text{if}&p\equiv 5\pmod{8}.\end{array}\right.$$

Now we consider the summation $\sum_{x\in A_3(p)}\chi(x)$. Similar to above we have that $A_3(p)$ is empty if $p\equiv 2\pmod{3}$, in which case our summation is zero. Also, if $p=3$ then our summation is $1$. Now, suppose that $p\equiv 1\pmod{3}$. Note, it is immediate that if $r\in A_3(p)$ then so is $r^2$. Similar to the previous situation, we have that $\chi(r)=1$ if and only if there is an element of order 6 in $\left(\mathbb{Z}/p\mathbb{Z}\right)^{\mathsf{x}}$, i.e., $p\equiv 1\pmod{6}$. Note, that as $p$ is prime, it follows that $p\equiv 1\pmod{6}$ is equivalent to $p\equiv 1\pmod{3}$. Thus, we have
$$\sum_{x\in A_3(p)}\chi(x)=\left\{\begin{array}{lcl}1&\text{if}&p=3\\0&\text{if}&p\equiv 2\pmod{3}\\2&\text{if}&p\equiv 1\pmod{3}.\end{array}\right.$$

We summarize our calculations in Table \ref{quadtab} in the Appendix.

\subsection{Eta-quotients}\label{EtaQuotient}
In this section we introduce the eta-function and present some results relating these to modular forms. For further details regarding the eta-function the interested reader is referred to \cite{Kohler}.

Recall Dedekind's eta-function given in (\ref{etadef}). The eta-function satisfies the following transformation properties with respect to our matrices $S,T$ defined in Section \ref{ModForm}
$$\eta(Sz)=\eta(-z^{-1})=\sqrt{-iz}\eta(z),\text{ }\eta(Tz)=\eta(z+1)=e^{\frac{2\pi i}{24}}\eta(z).$$
Combining these we can deduce the following general transformation formula
$$\eta(\gamma z)=\epsilon(\gamma)(cz+d)^{\frac{1}{2}}\eta(z)\text{ for all } \gamma=\begin{pmatrix}a&b\\c&d\end{pmatrix}\in\SL_2(\mathbb{Z}),$$
where 
$$\epsilon(\gamma)=\left\{\begin{array}{lcl}\left(\frac{d}{|c|}\right)e^{2\pi i\left(\frac{(a+d)c-bd(c^2-1)-3c}{24}\right)}&\text{if}&c\text{ odd}\\(-1)^{\frac{1}{4}(\sgn(c)-1)(\sgn(d)-1)}\left(\frac{d}{|c|}\right)e^{2\pi i\left(\frac{(a+d)c-bd(c^2-1)+3d-3-3cd}{24}\right)}&\text{if}&c\text{ even}\end{array}\right.$$
and $\sgn(x)=\frac{x}{|x|}$. In addition to the eta-function, we will also need to consider the related function $\eta(\delta z)$ for a positive integer $\delta$. If we set $f(z)=\eta(\delta z)$ then $f(z)$ satisfies
$$f(\gamma z)=\epsilon\left(\begin{pmatrix}a&\delta b\\c/\delta &d\end{pmatrix}\right)(cz+d)^{\frac{1}{2}}f(z)\text{, for all }\gamma=\begin{pmatrix}a&b\\c&d\end{pmatrix}\in\Gamma_0(\delta).$$
Finally, we will need the following transformation
$$f(Tz)=e^{\frac{2\pi i\delta}{24}}f(z).$$
Notice, that this function is ``almost'' a modular form. With this in mind, we consider certain products of these functions with the goal of eliminating the ``almost''. This brings us to eta-quotients, which we defined in (\ref{etaqdef}). We will be interested in when these eta-quotients are modular forms. We have the following theorem which partially answers this question.

\begin{theorem}\label{etacon}\cite[Thm. 1.64]{Ono1}
Define the an eta-quotient \[f(z) = \prod_{\delta|N} \eta^{r_\delta}(\delta z),\] and set $k = \frac{1}{2}\sum_{\delta|N} r_\delta\in\Z$. Suppose our exponents $r_1,...,r_N$ satisfy \begin{align*}
\sum_{\delta|N} \delta r_\delta &\equiv 0\pmod{24},\quad\text{and}\\
\sum_{\delta|N} \frac{N}{\delta} r_\delta &\equiv 0\pmod{24}.
\end{align*}
Then, \[f|_k\gamma(z) = \chi(d)f(z)\] for all $\gamma = \begin{pmatrix}a&b\\c&d\end{pmatrix}\in\Gamma_0(N)$, where $\chi(n) = \left(\frac{(-1)^k s}{n}\right)$ with $s = \prod_{\delta|N} \delta^{r_\delta}$.
\end{theorem}

This theorem provides conditions on when an eta-quotient is a weakly modular function. However, to answer the question of when an eta-quotient is a modular form we need the following theorem which provides information concerning the order of vanishing at the cusps of $\Gamma_0(N)$.

\begin{theorem}\label{ordvan}\cite[Thm. 1.65]{Ono1}
Let $f(z)$ be an eta-quotient satisfying the the conditions of Theorem \ref{etacon}. Let $c,d\in\N$ with $d|N$ and $(c,d) = 1$. Then, the order of vanishing of $f(z)$ at the cusp $\frac{c}{d}$ is \[v_d = \frac{N}{24}\sum_{\delta|N} \frac{(d,\delta)^2 r_\delta}{(d,\frac{N}{d})d\delta}.\]
\end{theorem}

\section{Proof of Results}
In this section, we will provide the proofs for the results given in Section 1. We will assume the eta-quotients being discussed always have $N = p>3$ a prime, unless otherwise stated. From Theorems \ref{etacon} and \ref{ordvan}, we have conditions that tell us when an eta-quotient is a holomorphic modular form. Thus, we will use the following equations: \begin{align}
\frac{1}{2}(r_1 + r_p) &= k\label{wteqn}\\
r_1+pr_p &\equiv 0 \pmod{24}\\
pr_1+r_p &\equiv 0\pmod{24}\\
v_1 &= \frac{1}{24}(pr_1+r_p) \label{uord}\\
v_p &= \frac{1}{24}(r_1+pr_p),\label{pord}
\end{align}
where $v_1$ and $v_p$ are the orders of vanishing at the two cusps of $\Gamma_1(p)$, $i\infty$ and $\frac{1}{p}$, respectively.

For a fixed prime $p$ and a fixed weight $k$, we see that it is possible to express $r_p$ in terms of $r_1$ by (\ref{wteqn}). It is convenient to rewrite (\ref{uord}) and (\ref{pord}) as \begin{align}
24v_1 &= 2k + (p-1)r_1\label{uordvan}\\
24v_p &=  2kp + (1-p)r_1.\label{pordvan}
\end{align} It is now clear that we can relate the orders of vanishing to the weight of an eta-quotient by \begin{equation}\label{ordvaneqn}
24(v_1+v_p) = 2k(p+1) .
\end{equation}
We begin the discussion for counting eta-quotients of level $\Gamma_0(p)$ by looking at possible conditions on $k$. These conditions were stated in Theorem \ref{kcond}, which we restate here for convenience.

\begin{theorem*}
Let $p>3$ be a prime and $k$ be an integer. Then there exists $f(z) = \eta^{r_1}(z)\eta^{r_p}(pz)$ such that $f(z)$ is a weakly holomorphic modular form with weight $k$ of level $\Gamma_1(p)$ if and only if $k$ is divisible by $h = \frac{1}{2}\gcd(p-1,24)$. 
\end{theorem*}

\begin{proof}
($\rightarrow$) Suppose that $f(z)\in M^!(\Gamma_1(p))$. We note that it suffices to show that we can satisfy (\ref{pordvan}) and (\ref{ordvaneqn}) since (\ref{uordvan}) can be gained from these two.

From (\ref{ordvaneqn}), we see that we want $\frac{k(p+1)}{12}$ to be an integer as the orders of vanishing, $v_1$ and $v_p$, are integers. From here we can find a divisor, $d$, of $k$ that would make this possible. Then by (\ref{pordvan}), we know that we need  $24|2kp - (p-1)r_1$. This gives us that \begin{equation}\label{kcondeqn}2pdn\equiv(p-1)r_1\pmod{24},\end{equation} where $dn=k$. Let $\delta = \gcd(24,2dp,p-1)$. Then, we get that $\frac{2dp}{\delta},\frac{p-1}{\delta}\in(\Z/(24/\delta)\Z)^\times$ and gains us our desired conclusion, where $d=h=\frac{1}{2}\gcd(p-1,24)$; except for when $p$ is congruent to $1$ or $17$ modulo 24.

Suppose that $p\equiv 1 \pmod{24}$. Then we can rewrite (\ref{kcondeqn}) as \[12n\equiv 0\pmod{24}.\] This tells us that $n$ must be even. Thus, we have that $k\equiv 0\pmod{12}$. We further note that $12 = \frac{1}{2}\gcd(24\ell,24)$. Thus, showing our result for this case.

Suppose that $p\equiv 17\pmod{24}$. Rewriting (\ref{kcondeqn}), we get \[68n \equiv 16r_1\pmod{24} .\] This tells us that \[5n\equiv 4r_1\pmod{6}.\] Since $5\in\Z/6\Z^\times$, we have \[n\equiv 2r_1\pmod{6}.\] Therefore, we know that $n$ must be even. Therefore, we have that $4|k$. As $4=\frac{1}{2}\gcd(24\ell + 16, 24)$, we reach our desired conclusion.

($\leftarrow$) Suppose that $h=\frac{1}{2}\gcd(p-1,24)$ divides $k$. We want to show that there exists $f(z) = \eta^{r_1}(z)\eta^{r_p}(pz)$ in $M^!_k(\Gamma_1(p))$. It is sufficient to show that there exists $r_1$ such that $r+p(2k-r_1)\equiv 0\pmod{24}$. We can interpret this as \[r_1(1-p) + 2pk = 24N,\] where $N\in\Z$. As $2h$ divides every term, we can get \[-r_1\left(\frac{p-1}{2h}\right) + p\frac{2k}{2h} = \frac{24}{2h}N.\] Therefore, we have that \[\frac{p-1}{2h}r_1 \equiv p\frac{k}{h}\pmod{\frac{24}{2h}}.\] Since $\frac{p-1}{2h}$ and $\frac{24}{2h}$ are relatively prime, $\frac{p-1}{2h}$ has an inverse in $\Z/(\frac{24}{2h})\Z$. Thus, there exists a unique $r_1 \in \Z/(\frac{24}{2h})\Z$ such that\[r_1 \equiv p\frac{k}{h}\left(\frac{p-1}{2h}\right)^{-1} \pmod{\frac{24}{2h}}.\]
\end{proof}

As mentioned in Section 1, we can extend Theorem \ref{kcond} to show when there exists $f(z) = \eta^{r_1}(z)\eta^{r_p}(pz)\in M_k(\Gamma_1(p))$. Before we do so, we need a lemma.

\begin{lemma}\label{funlem}
Let $N$ be an integer such that $\gcd(N,6) = 1$. Let $f(z)$ be given by \[f(z) = \prod_{d|N} \eta^{r_d}(dz).\] If $f\in M_k(\Gamma_0(N),\chi)$, then it must be the case that \[\sum_{d|N} dr_d\equiv 0\pmod{24},\quad\text{and}\quad \sum_{d|N} \frac{N}{d}r_d \equiv 0\pmod{24}.\]
\end{lemma}
\begin{proof}
Since $f\in M_k(\Gamma_0(p),\chi)$, the $q$-series expansion of $f$ about the cusp at infinity must look like \[f(z) = \sum_{n\geq 0} c_nq^n.\] Recall that $\eta(z) = q^{1/24} \prod_{n\geq 1} (1-q^n)$. Thus, we would have that \[\prod_{d|N} \eta^{r_d}(dz) = q^{\frac{\sum_{d|N} r_d d}{24}}\prod_{n\geq 1}\left(\prod_{d|N}(1-q^{dn})^{r_d}\right).\] Therefore, we need $24$ to divide \[\sum_{d|N} dr_d. \] We also note that for all primes $p\geq 5$, $p^2\equiv 1\pmod{24}$. Therefore, $Nd \equiv \frac{N}{d}\pmod{24}$. Thus, we have that \[0\equiv N\sum_{d|N} dr_d \equiv \sum_{d|N} \frac{N}{d}r_d \pmod{24}.\]
\end{proof}

As we wish to focus on holomorphic modular forms, we now want non-negative orders of vanishing, i.e. $v_1,v_p\geq 0$. Using this condition and (\ref{ordvaneqn}) we also have that $v_1,v_p\leq \frac{k(p+1)}{12}$. We use Figure \ref{graph} to show the line  that relates $v_1$ to $v_p$ given a fixed $k$ and $p$. We note that given equation (\ref{uordvan}), we can define $v_1$ in terms of $r_1$, and vice-versa. Thus, to count the number of eta-quotients of our desired form, it suffices to count the number of possible orders of vanishing. As orders of vanishing are integer values, we only consider integer points on the line given in Figure \ref{graph}.

\begin{figure}[h] \begin{center}
  \begin{tikzpicture}[scale=2, thick]
  \draw (0,0) --(2,0) node[midway,below]{$\frac{k(p+1)}{12}$} (2,0)node[right,right]{$v_1$}  -- (1,0) node[midway,below]{} -- (0,1) node[midway,above]{} -- (0,2) node[above,above]{$v_p$}
-- (0,0) node[midway,left]{$\frac{k(p+1)}{12}$};
  \end{tikzpicture}\end{center}
  \caption{The line $v_1 + v_p = \frac{k(p+1)}{12}$}\label{graph}
\end{figure}
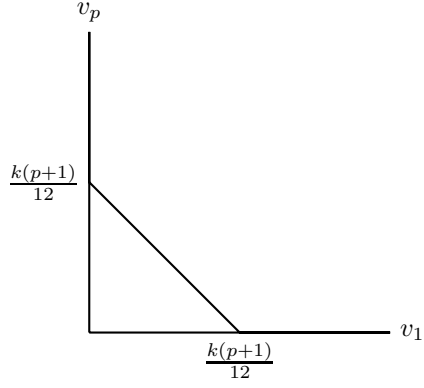

Furthermore, from equation (\ref{uordvan}), we have that \[(p-1)r_1 = 24v_1 - 2k.\] This implies that $24v_1 - 2k \equiv 0 \pmod{p-1}$. In other words, we can write $24v_1 - 2k = (p-1)\ell$ where $\ell\in\Z$. Recall how we defined $h$ in Theorem \ref{kcond}. Since $2h|2k$ and $2h|24$, we can write $\frac{24}{2h}v_1 - k' = \frac{p-1}{2h}\ell$. We also know that $2h|p-1$. Therefore, we have that \[\frac{24}{2h}v_1 \equiv k' \pmod{\frac{p-1}{2h}}.\] Since we have that $2h = \gcd (p-1, 24)$, we get that $1 = \gcd\left( \frac{p-1}{2h}, \frac{24}{2h}\right)$. This implies that we have a multiplicative inverse of $\frac{24}{2h}$ in $\Z/(\frac{p-1}{2h})\Z$. Thus, we have that \begin{equation}\label{ordcong} v_1\equiv \left(\frac{24}{2h}\right)^{-1}k'\pmod{\frac{p-1}{2h}}.\end{equation} Now we have the following corollary which follows from this explanation as well as Theorem \ref{kcond} and Lemma \ref{funlem}.

\begin{coro}
Let $p\geq 5$ be a prime. There exists $f(z) = \eta^{r_1}(z)\eta^{r_p}(pz)$ in $M_k(\Gamma_1(p)$ if and only if $h=\frac{1}{2}\gcd(p-1,24)$ divides $k$ and $\frac{p-1}{2h} \leq \frac{k(p+1)}{12}$.
\end{coro}

We note that by definition, cusp forms occur on the interior of our line, and non-cuspidal modular forms occur at the end points. For this reason it is useful to perform the counts of cusp forms and non-cuspidal modular forms separately. First, we prove the count of cusp forms given in Theorem \ref{etaPcount}, which we restate here for convenience.

\begin{theorem*}
Let $p>3$ be a prime. Let $k=hk'$ where $h$ is the needed divisor of $k$ given by Theorem \ref{kcond}. Let $c$ be the smallest positive integer representative of $\frac{k'h}{12}$ modulo $\frac{p-1}{2h}$.
\begin{enumerate}
\item For $c = \frac{k(p+1)}{12} - \floor{\frac{k(p+1)}{12\frac{p-1}{2h}}} \frac{p-1}{2h}$, the number of eta-quotients in $S_k(\Gamma_1(p))$ is \[\frac{k(p+1)}{12d}-1.\]
\item For $c < \frac{k(p+1)}{12} - \floor{\frac{k(p+1)}{12\frac{p-1}{2h}}} \frac{p-1}{2h}$, the number of eta-quotients in $S_k(\Gamma_1(p))$ is \[ \ceiling{\frac{k(p+1)}{12\frac{p-1}{2h}}}.\]
\item For $c > \frac{k(p+1)}{12} - \floor{\frac{k(p+1)}{12\frac{p-1}{2h}}} \frac{p-1}{2h}$, the number of eta-quotients in $S_k(\Gamma_1(p))$ is \[ \floor{\frac{k(p+1)}{12\frac{p-1}{2h}}}.\]
\end{enumerate}
\end{theorem*}
\begin{proof}
Since we are only considering cusp forms, we can assume that $v_1$, $v_p >0$. The number of points on our line from Figure \ref{graph} which satisfy this inequality and the congruence from (\ref{ordcong}) is the number of eta-quotients. We now consider three cases.
\begin{itemize}
\item Case 1: Suppose $c = 0 = \frac{k(p+1)}{12} - \floor{\frac{k(p+1)}{12\frac{p-1}{2h}}} \frac{p-1}{2h}$. Then, we have that $v_1 \equiv 0 \pmod{\frac{p-1}{2h}}$. Furthermore, we note that $\frac{p-1}{2h} \left| \frac{k(p+1)}{12}\right.$. Thus, we have that the number of points which match our congruence is $\frac{k(p+1)}{12\frac{p-1}{2h}}$. However, we will note that one of these points gives us $v_p=0$, which is not desired. Therefore, the number of eta-quotients that are in $S_k(\Gamma_1(p))$ is \[\frac{k(p+1)}{12\frac{p-1}{2h}}-1.\]
\item Case 2: Suppose $c<\frac{k(p+1)}{12} - \floor{\frac{k(p+1)}{12\frac{p-1}{2h}}} \frac{p-1}{2h}$. Then we note that $\floor{\frac{k(p+1)}{12\frac{p-1}{2h}}}\frac{p-1}{2h}$ is less than $\frac{k(p+1)}{12}$. However, since $c<\frac{k(p+1)}{12} - \floor{\frac{k(p+1)}{12\frac{p-1}{2h}}} \frac{p-1}{2h}$, we have another point to count that is between $\floor{\frac{k(p+1)}{12\frac{p-1}{2h}}}\frac{p-1}{2h}$ and $\frac{k(p+1)}{12}$. Therefore, the number of eta-quotients that are in $S_k(\Gamma_1(p))$ is \[\ceiling{\frac{k(p+1)}{12\frac{p-1}{2h}}}.\]
\item Case 3: Suppose $c>\frac{k(p+1)}{12} - \floor{\frac{k(p+1)}{12\frac{p-1}{2h}}} \frac{p-1}{2h}$. Then we note that $\floor{\frac{k(p+1)}{12\frac{p-1}{2h}}}\frac{p-1}{2h}$ is less than $\frac{k(p+1)}{12}$. Since $c>\frac{k(p+1)}{12} - \floor{\frac{k(p+1)}{12\frac{p-1}{2h}}} \frac{p-1}{2h}$, we have no more points to count between $\floor{\frac{k(p+1)}{12\frac{p-1}{2h}}}\frac{p-1}{2h}$ and $\frac{k(p+1)}{12}$. Therefore, the number of eta-quotients that are in $S_k(\Gamma_1(p))$ is \[\floor{\frac{k(p+1)}{12\frac{p-1}{2h}}}.\]
\end{itemize}
\end{proof}
Second, we prove the count of non-cusp forms given in Theorem \ref{noncuspCount}, which we restate here for convenience.
\begin{theorem*}
Let $p>3$ be a prime. Then, $M_k(\Gamma_0(p))\setminus S_k(\Gamma_1(p))$ contains at least one eta-quotient if and only if $\frac{p-1}{2} | k$. Furthermore, for $k>0$ and $\frac{p-1}{2}|k$, there are exactly two eta-quotients in $M_k(\Gamma_1(p))\setminus S_k(\Gamma_1(p))$, which are of the form \[ \frac{\etaq{p}{\frac{2pk}{p-1}}}{\etaq{}{\frac{2k}{p-1}}},\quad\text{and}\quad \frac{\etaq{}{\frac{2pk}{p-1}}}{\etaq{p}{\frac{2k}{p-1}}}.\]
\end{theorem*}
\begin{proof}
($\rightarrow$) Suppose $f(x) \in M_k(\Gamma_1(p))\setminus S_k(\Gamma_1(p))$ is an eta-quotient satisfying Theorem \ref{etacon}. Then, we know that at least one of the orders of vanishing must be zero. Thus, we have two cases.
 \begin{itemize}
\item Case 1: Suppose $v_1 = 0$. Then, $pr_1 + r_p = 0$, which can be rewritten to get $(p-1)r_1 = -2k$. Therefore, we have that $\frac{p-1}{2} |k$. Furthermore, we can get that $r_1 = \frac{2k}{p-1}$ and thus, $r_p = \frac{2pk}{p-1}$. When plugging these values into $v_p$ we get \[v_p = \frac{1}{24}\left( \frac{-2k}{p-1} + \frac{2pk}{p-1}\right) = \frac{2k}{24} >0.\]
\item Case 2: Suppose $v_p = 0$. Then, $r_1 + pr_p = 0$, which can be rewritten to get $(1-p)r_1 = -2pk$. Therefore, $\frac{p-1}{2}|k$ since $p\not| (p-1)$ and therefore $p|r_1$. Furthermore, we get that $r_1 = \frac{2pk}{p-1}$, and thus $r_p = \frac{-2k}{p-1}$. When plugging these values into $v_1$ we get \[v_1 = \frac{1}{24}\left( \frac{2pk}{p-1} + \frac{-2k}{p-1}\right) = \frac{2k}{24} > 0.\]
\end{itemize}
In both cases, the number needed to divide $k$ is the same. Furthermore, both create a single eta-quotient for a fixed $k$. Therefore, we have that $\frac{p-1}{2} |k$. Furthermore, there are exactly two eta-quotients which result from looking at either of the orders of vanishing being zero, and they are \[ \frac{\etaq{p}{\frac{2pk}{p-1}}}{\etaq{}{\frac{2k}{p-1}}},\quad\text{and}\quad \frac{\etaq{}{\frac{2pk}{p-1}}}{\etaq{p}{\frac{2k}{p-1}}}.\]

($\leftarrow$) Suppose that $k = \frac{p-1}{2}m>0$ for some integer $m$. Also, suppose we have the two eta-quotients \[ \frac{\etaq{p}{\frac{2pk}{p-1}}}{\etaq{}{\frac{2k}{p-1}}},\quad\text{and}\quad \frac{\etaq{}{\frac{2pk}{p-1}}}{\etaq{p}{\frac{2k}{p-1}}}.\] We consider each eta-quotient as its own case. \begin{itemize}
\item Consider $\frac{\etaq{p}{\frac{2pk}{p-1}}}{\etaq{}{\frac{2k}{p-1}}}$. We note that $r_1+r_p = \frac{-2k}{p-1}+\frac{2pk}{p-1} = 2k$. Furthermore, \[r_1 +pr_p = \frac{-2k}{p-1}+\frac{2pk}{p-1}p = (p^2-1)m \equiv 0 \pmod{24}\] since $p$ is relatively prime to $24$; and \[pr_1+r_p = p\frac{-2k}{p-1}+\frac{2pk}{p-1} = 0 \equiv 0\pmod{24}.\] When looking at that orders of vanishing, we get \[v_1 = \frac{1}{24}(pr_1+r_p) = \frac{1}{24}\left(p\frac{-2k}{p-1}+\frac{2pk}{p-1}\right) = 0 \geq 0,\quad\text{and}\] \[v_p = \frac{1}{24}(r_1 +pr_p) = \frac{1}{24}\left(\frac{-2k}{p-1}+\frac{2pk}{p-1}p\right) = \frac{1}{24}(p^2-1)m\geq 0.\] Since our orders of vanishing are both $\geq 0$ and one of them is equal to $0$, we have that $\frac{\etaq{p}{\frac{2pk}{p-1}}}{\etaq{}{\frac{2k}{p-1}}} \in M_k(\Gamma_1(p))\setminus S_k(\Gamma_1(p))$.
\item Consider $\frac{\etaq{}{\frac{2pk}{p-1}}}{\etaq{p}{\frac{2k}{p-1}}}$. We note that $r_1+r_p = \frac{2pk}{p-1}+\frac{-2k}{p-1} = 2k$. Furthermore, \[r_1 +pr_p = \frac{2pk}{p-1}+p\frac{-2k}{p-1} = 0 \equiv 0 \pmod{24};\] and since $p$ is relatively prime to $24$, \[pr_1+r_p = p\frac{2pk}{p-1}+\frac{-2k}{p-1} = (p^2-1)m \equiv 0\pmod{24}.\] When looking at that orders of vanishing, we get \[v_1 = \frac{1}{24}(pr_1+r_p) = \frac{1}{24}\left(p\frac{2pk}{p-1}+\frac{-2k}{p-1}\right) = \frac{1}{24}(p^2-1)m \geq 0,\quad\text{and}\] \[v_p = \frac{1}{24}(r_1 +pr_p) = \frac{1}{24}\left(\frac{2pk}{p-1}+\frac{-2k}{p-1}p\right) = 0\geq 0.\] Since our orders of vanishing are both $\geq 0$ and one of them is equal to $0$, we have that 
$$\frac{\etaq{}{\frac{2pk}{p-1}}}{\etaq{p}{\frac{2k}{p-1}}} \in M_k(\Gamma_1(p))\setminus S_k(\Gamma_1(p)).$$
\end{itemize} Thus, $M_k(\Gamma_1(p))\setminus S_k(\Gamma_1(p))$ contains exactly two eta-quotients.
\end{proof}
From the eta-quotients given in the theorem, let $k=\frac{p-1}{2}m$ where $m$ is a positive integer. Then the eta-quotients have characters \[\chi_1(n) = \left(\frac{(-1)^{\frac{p-1}{2}m}p^{pm}}{n}\right),\quad\text{and}\quad \chi_2(n) = \left(\frac{(-1)^{\frac{p-1}{2}m}p^m}{n}\right),\] respectively. In the case where $m$ is even, both of the characters are guaranteed to be the trivial character. When $m$ is odd, we are guaranteed to have a quadratic character. 

Now that we know how many eta-quotients there are and can write down what they are if needed, we would like the dimension of the space spanned by these eta-quotients. This is provided by Theorem \ref{plinind}, which we restate here for convenience.
\begin{theorem*}
Let $p>3$ be a prime. Then, the eta-quotients in $M_k(\Gamma_1(p))$ given by the previous theorems are linearly independent.
\end{theorem*}
\begin{proof}
Suppose that we are looking at eta-quotients in $M_k(\Gamma_1(p))$ for $p>3$ a prime. Without loss of generality, we look at the Fourier about the cusp at $\infty$. By using the Sturm Bound \cite{Ono1}, we get that we need to compare the first $\floor{\frac{pk}{12}}+1$ terms of each Fourier series. We can pick a cusp and order the eta-quotients increasingly by looking at the order of vanishing. We can then create a matrix, $A$, where the $i,j$-th entry represents $a(j)$ in the $i$-th eta-quotient's Fourier series. Since all of the eta-quotients have different orders of vanishing and they are in increasing order, we get that $A$ is in echelon form. Furthermore, none of the rows change to non-zero  at the same time. This tells us that all the rows are linearly independent. Thus all of the eta-quotients are linearly independent.
\end{proof}
The following corollaries can all be obtained by comparing dimension formulas with our counts and applying the previous theorem.
\begin{coro}
Let $p\geq 5$ be a prime. Denote the space of level $p$, weight $k$ eta-quotients by $\eta_k(p)$.
\begin{enumerate}
\item
If $p\equiv 3\pmod{4}$, then taking the limit over odd $k$ in the appropriate congruence class from Theorem \ref{kcond} we have
$$\lim_{k\rightarrow \infty}\frac{\dim\eta_k(p)}{\dim S_k\left(p,\left(\frac{\cdot}{p}\right)\right)}=\frac{2h}{p-1}.$$
\item
If $p\equiv 3\pmod{4}$, then taking the limit over even $k$ in the appropriate congruence class from Theorem \ref{kcond} we have
$$\lim_{k\rightarrow \infty}\frac{\dim\eta_k(p)}{\dim S_k(\Gamma_0(p))}=\frac{2h}{p-1}.$$
\item
If $p\equiv 1\pmod{4}$, then taking the limit over all $k$ in the appropriate congruence class from Theorem \ref{kcond} we have
$$\lim_{k\rightarrow \infty}\frac{\dim\eta_k(p)}{\dim S_k(\Gamma_0(p))+\dim S_k\left(p,\left(\frac{\cdot}{p}\right)\right)}=\frac{h}{p-1}.$$
\end{enumerate}
\end{coro}
Finally, we would like to consider the case that our $v_1$ and $v_p$ are integral but do not correspond to integral $r_1,r_p$. To gain some intuition concerning the properties of the ``eta-quotients'' formed from these $r_1,r_p$ we consider the following example.
\begin{example}
Let $p=11$ and $k=6$. Note, in this situation we have that in order to have eta-quotients we must have $v_1\equiv 3\pmod{5}$. So, we will investigate the properties of the function obtained by choosing $v_1\not\equiv 3\pmod{5}$.

Consider $v_1=1$. This implies $v_p=5$.
Then,
$$\begin{pmatrix}r_1\\r_p\end{pmatrix}=\begin{pmatrix}1&11\\11&1\end{pmatrix}^{-1}\begin{pmatrix}24\\120\end{pmatrix}=\begin{pmatrix}54/5\\6/5\end{pmatrix}.$$
Now, we can use these to form the ``eta-quotient''
$$f(z)=\eta^{\frac{54}{5}}(z)\eta^{\frac{6}{5}}(11z).$$
Using the transformation properties from Section \ref{EtaQuotient} we have that
\begin{align*}
f(Tz)&=e^{\frac{27\pi i}{30}}\eta^{\frac{54}{5}}(z)e^{\frac{11\pi i }{10}}\eta^{\frac{6}{5}}(11z)\\
&=f(z).
\end{align*}
Note, if we raise $f(z)$ to the $5^{th}$ power to cancel the denominators of the $r_1$ and $r_p$ then we can use Theorem \ref{etacon} to verify that we obtain $f(z)^5\in S_{30}(\Gamma_0(11))$, i.e., our lattice point corresponds to a ``root'' of an $\eta$ quotient of higher weight.

Note, the remaining choices for $v_1$ give us similar results.
\end{example}

\section{Conclusion and Further Questions}
The paper details the number of eta-quotients in $M_k(\Gamma_0(p))$ of the form $\eta^{r_1}(z)\eta^{r_p}(pz)$. It is strongly suspected that these are not all the eta-quotients, but at the moment it is unclear what these other eta-quotients look like. Further work in this project would involve generalizing these results for all levels as well as figuring out the other forms of eta-quotients that would be possible.

\section*{Acknowledgments}
The first author would like to thank Holly Swisher, for aid with editing and organization, and Dania Zantout, for help with understanding of background material.
\section*{Appendix: Tables}

\begin{center}
\begin{longtable}{|c|*{4}{c|}}\hline
\multicolumn{5}{|c|}{$\dim S_k(\Gamma_0(p))$} \\
\hline
\backslashbox{$k(12)$}{$p(12)$}&1&5&7&11\\
\hline
0&$\frac{(p+1)(k-1)+2}{12}$&$\frac{(p+1)(k-1)-6}{12}$&$\frac{(p+1)(k-1)-4}{12}$&$\frac{(p+1)(k-1)-12}{12}$\\
\hline
1&0&0&0&0\\
\hline
2&$\frac{(p+1)(k-1)-26}{12}$&$\frac{(p+1)(k-1)-18}{12}$&$\frac{(p+1)(k-1)-20}{12}$&$\frac{(p+1)(k-1)-12}{12}$\\
\hline
3&0&0&0&0\\
\hline
4&$\frac{(p+1)(k-1)-6}{12}$&$\frac{(p+1)(k-1)-6}{12}$&$\frac{(p+1)(k-1)-12}{12}$&$\frac{(p+1)(k-1)-12}{12}$\\
\hline
5&0&0&0&0\\
\hline
6&$\frac{(p+1)(k-1)-10}{12}$&$\frac{(p+1)(k-1)-18}{12}$&$\frac{(p+1)(k-1)-4}{12}$&$\frac{(p+1)(k-1)-12}{12}$\\
\hline
7&0&0&0&0\\
\hline
8&$\frac{(p+1)(k-1)-14}{12}$&$\frac{(p+1)(k-1)-6}{12}$&$\frac{(p+1)(k-1)-20}{12}$&$\frac{(p+1)(k-1)-12}{12}$\\
\hline
9&0&0&0&0\\
\hline
10&$\frac{(p+1)(k-1)-18}{12}$&$\frac{(p+1)(k-1)-18}{12}$&$\frac{(p+1)(k-1)-12}{12}$&$\frac{(p+1)(k-1)-12}{12}$\\
\hline
11&0&0&0&0\\
\hline
\caption{The dimension of $S_k(\Gamma_0(p))$ with trivial character and $k>2$.}\label{trivtab}
\end{longtable}
\end{center}

\begin{table}[h]
\begin{center}
\begin{tabular}{|c|*{4}{c|}}\hline
\multicolumn{5}{|c|}{$\dim S_k(p,\left(\frac{\cdot}{p}\right))$} \\
\hline
\backslashbox{$k(12)$}{$p(24)$}&1&5&7&11\\
\hline
0&$\frac{(k-1)(p+1)+8}{12}$&$\frac{(k-1)(p+1)-12}{12}$&0&0\\
\hline
1&0&0&$\frac{(k-1)(p+1)}{12}$&$\frac{(k-1)(p+1)-6}{12}$\\
\hline
2&$\frac{(k-1)(p+1)-20}{12}$&$\frac{(k-1)(p+1)}{12}$&0&0\\
\hline
3&0&0&$\frac{(k-1)(p+1)+2}{12}$&$\frac{(k-1)(p+1)-6}{12}$\\
\hline
4&$\frac{(k-1)(p+1)}{12}$&$\frac{(k-1)(p+1)-12}{12}$&0&0\\
\hline
5&0&0&$\frac{(k-1)(p+1)-14}{12}$&$\frac{(k-1)(p+1)-6}{12}$\\
\hline
6&$\frac{(k-1)(p+1)-4}{12}$&$\frac{(k-1)(p+1)}{12}$&0&0\\
\hline
7&0&0&$\frac{(k-1)(p+1)}{12}$&$\frac{(k-1)(p+1)-6}{12}$\\
\hline
8&$\frac{(k-1)(p+1)-4}{12}$&$\frac{(k-1)(p+1)-12}{12}$&0&0\\
\hline
9&0&0&$\frac{(k-1)(p+1)+2}{12}$&$\frac{(k-1)(p+1)-6}{12}$\\
\hline
10&$\frac{(k-1)(p+1)-12}{12}$&$\frac{(k-1)(p+1)}{12}$&0&0\\
\hline
11&0&0&$\frac{(k-1)(p+1)-14}{12}$&$\frac{(k-1)(p+1)-6}{12}$\\
\hline
\end{tabular}

\begin{tabular}{|c|*{4}{c|}}\hline
\multicolumn{5}{|c|}{$\dim S_k(p,\left(\frac{\cdot}{p}\right))$} \\
\hline
\backslashbox{$k(12)$}{$p(24)$}&13&17&19&23\\
\hline
0&$\frac{(k-1)(p+1)-4}{12}$&$\frac{(k-1)(p+1)}{12}$&0&0\\
\hline
1&0&0&$\frac{(k-1)(p+1)}{12}$&$\frac{(k-1)(p+1)-6}{12}$\\
\hline
2&$\frac{(k-1)(p+1)-8}{12}$&$\frac{(k-1)(p+1)-12}{12}$&0&0\\
\hline
3&0&0&$\frac{(k-1)(p+1)+2}{12}$&$\frac{(k-1)(p+1)-6}{12}$\\
\hline
4&$\frac{(k-1)(p+1)-12}{12}$&$\frac{(k-1)(p+1)}{12}$&0&0\\
\hline
5&0&0&$\frac{(k-1)(p+1)-14}{12}$&$\frac{(k-1)(p+1)-6}{12}$\\
\hline
6&$\frac{(k-1)(p+1)+8}{12}$&$\frac{(k-1)(p+1)-12}{12}$&0&0\\
\hline
7&0&0&$\frac{(k-1)(p+1)}{12}$&$\frac{(k-1)(p+1)-6}{12}$\\
\hline
8&$\frac{(k-1)(p+1)-20}{12}$&$\frac{(k-1)(p+1)}{12}$&0&0\\
\hline
9&0&0&$\frac{(k-1)(p+1)+2}{12}$&$\frac{(k-1)(p+1)-6}{12}$\\
\hline
10&$\frac{(k-1)(p+1)}{12}$&$\frac{(k-1)(p+1)-12}{12}$&0&0\\
\hline
11&0&0&$\frac{(k-1)(p+1)-14}{12}$&$\frac{(k-1)(p+1)-6}{12}$\\
\hline
\end{tabular}
\caption{Dimension of $S_k(p, \left(\frac{\cdot}{p}\right))$.}\label{quadtab}
\end{center}
\end{table}
\bigskip\bigskip

\bibliographystyle{alpha}

\begin{thebibliography}{Ono04}

\bibitem[Kob93]{Kob1}
Neal Koblitz.
\newblock {\em Introduction to elliptic curves and modular forms}, volume~97 of
  {\em Graduate Texts in Mathematics}.
\newblock Springer-Verlag, New York, second edition, 1993.

\bibitem[K{\"o}h11]{Kohler}
G{\"u}nter K{\"o}hler.
\newblock {\em Eta products and theta series identities}.
\newblock Springer monographs in mathematics. Springer-Verlag, New York, 2011.

\bibitem[Ono04]{Ono1}
Ken Ono.
\newblock {\em The web of modularity: arithmetic of the coefficients of modular
  forms and {$q$}-series}, volume 102 of {\em CBMS Regional Conference Series
  in Mathematics}.
\newblock Published for the Conference Board of the Mathematical Sciences,
  Washington, DC, 2004.

\bibitem[RW15]{Rou}
Jeremy Rouse and John~J. Webb.
\newblock On spaces of modular forms spanned by eta-quotients.
\newblock {\em Adv. Math.}, 272:200--224, 2015.

\bibitem[Ste07]{Ste1}
William Stein.
\newblock {\em Modular forms, a computational approach}, volume~79 of {\em
  Graduate Studies in Mathematics}.
\newblock American Mathematical Society, Providence, RI, 2007.
\newblock With an appendix by Paul E. Gunnells.

\end{thebibliography}

\end{document}